\documentclass[letterpaper, 12pt]{amsart}

\usepackage{fullpage}
\usepackage{graphicx}
\usepackage{amsmath, amssymb, amsfonts, amsthm}
\usepackage[foot]{amsaddr}
\usepackage{mathtools, mathrsfs}
\usepackage{tikz}

\usepackage{color}
\definecolor{darkblue}{rgb}{0,0,0.4}
\usepackage[plainpages=false,colorlinks=true,citecolor=blue,filecolor=black,linkcolor=red,urlcolor=darkblue]{hyperref}

\usepackage{url}

\usepackage{cancel}
\usepackage{enumerate}
\usepackage{empheq}

\newtheorem{definition}{Definition}[section]
\newtheorem{theorem}[definition]{Theorem}

\newtheorem{lemma}[definition]{Lemma}
\newtheorem{corollary}[definition]{Corollary}
\newtheorem{remark}{Remark}

\newcommand{\bond}{\mathrm{Bo}}

\newcommand{\R}{\mathbb{R}}

\newcommand{\N}{\mathbb{N}}
\renewcommand{\H}{\mathbf{H}}

\newcommand{\D}{\mathcal{D}}
\newcommand{\B}{\mathcal{B}}
\newcommand{\F}{\mathcal{F}}

\renewcommand{\O}{\mathcal{O}}
\renewcommand{\o}{o}
\newcommand{\FF}{\mathscr{F}}
\newcommand{\x}{\mathbf{x}}

\renewcommand{\u}{\mathbf{u}}
\newcommand{\n}{\mathbf{\hat{n}}}

\newcommand{\z}{\mathbf{\hat{z}}}
\newcommand{\X}{\mathbf{\hat x}}
\newcommand{\y}{\mathbf{\hat y}}
\allowdisplaybreaks

\begin{document}
\title{A Variational Characterization of \\ Fluid Sloshing with Surface Tension}
\author{Chee Han Tan}
\author{Christel Hohenegger}
\author{Braxton Osting}
\address{Department of Mathematics, University of Utah, Salt Lake City, UT 84112}
\email{tan@math.utah.edu, choheneg@math.utah.edu, and osting@math.utah.edu}

\keywords{fluid sloshing, surface tension, contact angle, calculus of variations}
\subjclass[2010]{49R05, 76M30, 76B45}  

\date{\today}

\begin{abstract}
We consider the sloshing problem for an incompressible, inviscid, irrotational fluid in an open container, including effects due to surface tension on the free surface. We restrict ourselves to a constant contact angle and seek time-harmonic solutions of the linearized problem, which describes the time-evolution of the fluid due to a small initial disturbance of the surface at rest. As opposed to the zero surface tension case, where the problem reduces to a partial differential equation for the velocity potential, we obtain a coupled system for the velocity potential and the free surface displacement. We derive a new variational formulation of the coupled problem and establish the existence of solutions using the direct method from the calculus of variations. 
We prove a domain monotonicity result for the fundamental sloshing eigenvalue. In the limit of zero surface tension, we recover the variational formulation of the  mixed Steklov-Neumann eigenvalue problem and  give the first-order perturbation formula for a simple eigenvalue. 
\end{abstract}

\maketitle

\section{Introduction}
In fluid dynamics, sloshing refers to the motion of the free surface of a liquid inside a container. 
Spilling and splashing of a fluid is possible if the sloshing amplitude is large enough. 
Indeed, sloshing of a cup of coffee can devastate a perfectly good day \cite{Mayer:2012aa}. Examples of more significant consequences due to sloshing  include the free surface effect in ships and trucks transporting oil and liquified natural gas (LNG) \cite{Abramson:1976aa, Faltinsen:2009aa} and sloshing of liquid propellant in spacecraft tanks and rockets \cite{Abramson:1966aa, Ibrahim:2005aa}.  

LNG carriers usually operate either fully loaded or nearly empty, but there has been a growing demand for membrane-type LNG carriers that can operate with cargo loaded to any filling level. Experimental and numerical studies show that the coupling effect between sloshing dynamics inside tanks and ship motions can be significant at certain frequencies of partially filled tanks, where violent sloshing generates high impact pressure on the tank surfaces and compromises structural safety. As such, prediciting and understanding the natural sloshing frequencies, modes, and impact load at partially filled levels are of great concern to the safety and operability of LNG carriers close to an LNG terminal and remain one of the most crucial design aspects in LNG cargo containment system. 

Since Robert Goddard's first launch of a liquid propellant rocket in 1926, scientists and engineers have worked to better understand the sloshing behavior of propellants in their tanks. This is important not only in terms of reducing costs and increasing efficiency of future spacecraft designs but also in minimizing potential impacts especially on flight safety, since violent sloshing fuels can, for example, produce highly localized impact loads and pressure on tank walls or affect the spacecraft's guidance system. There are many instances  where space missions were either deemed a failure or could not be completed due to sloshing \cite{Kyle:1957aa, Hydro:1965aa, Jones:1995aa, Strikwerda:2001aa, Schlee:2007aa}. For instance, in March 2007, the SpaceX Falcon 1 vehicle tumbled out of control, when an oscillation appeared in the upper stage control system approximately 90 seconds into the burn and instability grew in pitch. It was verified by third party industry experts that cryogenic liquid oxygen (LOX) sloshing was the primary contributor to this instability \cite{Falcon1:2007aa}.

Recent advances in computational fluid dynamics (CFD) tools have made accurate numerical modeling of sloshing dynamics and extraction of mechanical parameters such as sloshing frequency and sloshing mass center possible \cite{Peugeot:2010aa}. However, it requires extensive experimental validation and verification in microgravity or zero gravity environment, since fluid behaves in an unpredictable manner due to the absence of gravity. To benchmark and expand CFD tools to characterize sloshing dynamics, engineers with NASA together with researchers from the Florida Institutute of Technology and the Massachusetts Institute of Technology designed the SPHERES-Slosh experiment (SSE), carried aboard at the International Space Station. This investigation is planned to collect valuable data and information on how liquids move around inside of a container in the presence of external force. A description of design details of the SSE can be found in \cite{Chintalapati:2013aa, Lapilli:2015aa}. 

In this paper, we study the linearized sloshing problem of an incompressible, inviscid, irrotational fluid in containers, including surface tension effects on the free surface.

\subsection{Surface tension effects.}\label{sec:SurfTenEff}
Surface tension is present at all fluid interfaces, and it manifests itself in nature, most commonly in capillary phenomena such as the rise of water up a capillary tube. Surface tension, defined as force per unit length, can be explained in terms of surface force or surface energy \cite{Leal:2007aa, Bush:2013aa}. Roughly speaking, it is the intermolecular force required to contract the liquid surface to its minimal surface area. Geometrically, including surface tension forces is equivalent to considering the curvature of the interface. If we denote by $\rho,g,T,l$ the density of the fluid, gravitational acceleration, surface tension, and some characteristic length scale of the system respectively, and assume that $\rho,T$ are constant, then the dimensionless parameter $\bond = \rho gl^2/T $, known as the Bond-E\"otv\"os number \cite{Ibrahim:2005aa}, measures the importance of the surface tension force relative to the gravitational force. For $\bond\gg 1$, surface tension is assumed to be negligible and this is often the case for fluids in large containers under a regular gravity field. However,  if $\bond\ll 1$, then surface tension is not negligible anymore; this occurs when one is examining sloshing behavior in a microgravity environment or if the characteristic length of the interface is much smaller compared to the capillary length $l_c^2=T/(\rho g)$. 

Closely related to the concept of surface tension is that of the contact angle, in other words the angle of contact between the solid and the liquid-air interface along the line of intersection between the container's wall and the fluid free surface, known as the contact line \cite{Dussan:1979aa}. On one hand, the contact angle is a geometrical quantity uniquely defined as a dot product, while on the other hand it is a physical quantity which quantifies the wettability of a solid surface. In the static case, the resulting condition is known as the Young's equation and can be derived from an energy minimization argument on the contact line \cite{Moiseyev:1968aa, Finn:1986aa}. In the dynamic case, accurately describing the contact angle remains poorly understood, mainly due to contact angle hysteresis. We will further discuss the contact angle  in Section \ref{sec:ContactAngle}.

\subsection{Sloshing problem with surface tension.}\label{sec:Setup}
Consider an irrotational flow of an incompressible, inviscid fluid occupying a bounded region $\D_T\subset\R^3$ in a simply connected container. The Cartesian coordinates $\tilde\x = (\tilde x,\tilde y,\tilde z)$ are chosen in such a way that the static free surface (or static meniscus), denoted by $\F$, lies in the $\tilde x$-$\tilde y$ plane and the $\tilde z$-axis is directed upward.  Here, $\D_T$ is a bounded simply connected Lipschitz domain; in particular its boundary $\partial\D_T$ has no cusps. $\partial\D_T$ consists of two parts: the (evolving) free surface $\F_T$ defined by
$ \F_T = \{(\tilde x,\tilde y,\tilde z)\in\R^3\colon\tilde z=\tilde\eta(\tilde x,\tilde y,\tilde t)\}, $
where $\tilde\eta$ is the free surface displacement, together with the wetted boundary $\B=\partial\D_T\setminus\F_T$. Moreover, the container's wall over which the contact line moves is vertical. The subscripts on $\D_T, \F_T$ are used to denote  time-dependence. 

The static meniscus $\F$ is assumed to intersect the vertical container wall orthogonally and this corresponds to a 90$^\circ$ (static) contact angle and, together with the assumption that the wall is vertical near the free surface,  implies  $\n_\B(x) = \n_{\partial\F}(x)$ for all $x\in\partial\F$; see Figure \ref{f:diagram}.  Another consequence is that $\F$ is a flat interface on the plane $\{\tilde z=0\}$; this will be proved in Section \ref{sec:ContactAngle}. One can think of $\F_T$ as a small perturbation of $\F$.

We give a brief description of the water waves equations describing fluid motion in $\D_T$; details of the derivation can be found in Appendix \ref{sec:AppA}. We denote by $\tilde \u(\tilde\x,\tilde t)$ the velocity field of the fluid. Incompressibility and irrotationality imply the existence of a velocity potential, denoted $\tilde\phi=\tilde\phi(\tilde\x,\tilde t)$, satisfying Laplace's equation in $\D_T$. The Neumann boundary condition is imposed on $\B$, while the classical kinematic and dynamic boundary conditions are imposed on $\F_T$, the latter of which can be expressed in terms of $\tilde\phi$ using Bernoulli's principle for an ideal fluid with unsteady irrotational flow. Nondimensionalizing the system with dimensionless variables
\begin{equation} \label{eq:NonDim} 
\x = \dfrac{\tilde\x}{a}, \quad t = \sqrt{\dfrac{g}{a}}\tilde t, \quad\phi = \dfrac{\tilde\phi}{a\sqrt{ga}}, \quad\eta = \dfrac{\tilde\eta}{a}, 
\end{equation}
where $a>0$ is some characteristic length scale of the system, we obtain the following system of dimensionless nonlinear partial differential equations:

\begin{subequations} \label{eq:Slosh1}
\begin{alignat}{2}
\Delta\phi &=0 && \qquad\textrm{ in }\D_T, \\ 
\partial_\n\phi &= 0 && \qquad\textrm{ on }\B, \\
\eta_t + \nabla\phi\cdot\nabla (\eta - z) &= 0 && \qquad\textrm{ on }\F_T, \label{eq:Slosh1c}\\
\phi_t + \frac{1}{2}|\nabla\phi|^2 + \eta &= -\frac{1}{\bond}\nabla\cdot\n_{\F_T} && \qquad\textrm{ on }\F_T, \label{eq:Slosh1d}\\
\n_\B\cdot\n_{\F_T} & =0 && \qquad\textrm{ on }\partial\F_T. \label{eq:Slosh1e}
\end{alignat}
\end{subequations}
Here, $\nabla = (\partial_x,\partial_y,\partial_z)$, $\Delta = \partial_x^2 + \partial_y^2 + \partial_z^2$, $\n_\B,\n_{\F_T}$ are the outward unit normal to the boundary $\B$ and the free surface $\F_T$, respectively. By $\partial_n$ we mean the normal derivative of a function. We discuss in details the contact line boundary condition \eqref{eq:Slosh1e} in Section \ref{sec:ContactAngle}.

\begin{figure}[t]
\centering
\begin{tikzpicture}[thick, scale=0.5]
\draw plot[smooth cycle] coordinates {(0,0) (0.5,-0.5) (2,-0.3) (4,-0.5) (5,0) (4.5,0.6) (3,0.5) (1.2,0.7) (0.5,0.5)};
\draw (0,0)--(0,-3.6);
\draw (5,0)--(5,-3.6);
\draw [dashed] (0.5,0.5)--(0.5,-3.1);
\draw [dashed] (4.5,0.6)--(4.5,-3);
\draw [fill=gray!50] plot[smooth] coordinates {(0,-1.8) (0.5,-2.3) (2,-2.1) (4,-2.3) (5,-1.8)};
\draw [dashed, fill=gray!50] plot[smooth] coordinates {(5,-1.8) (4.5,-1.2) (3,-1.3) (1.2,-1.1) (0.5,-1.3) (0,-1.8)};
\draw plot[smooth] coordinates {(0,-3.6) (0.5,-4.1) (2,-3.9) (4,-4.1) (5,-3.6)};
\draw [dashed] plot[smooth] coordinates {(5,-3.6) (4.5,-3) (3,-3.1) (1.2,-2.9) (0.5,-3.1) (0,-3.6)};
\draw [dashed] plot[smooth] coordinates {(0.5,-3.1) (0.5,-4) (0.8,-5)};
\draw [dashed] plot[smooth] coordinates {(4.5,-3) (4.6,-4.3) (4.5, -5) (4,-5.6)};
\draw (0,-3.6) to[out=-90,in=-170] (0,-7);
\draw (5,-3.6) to[out=-90,in=-30] (5,-7);
\draw plot[smooth] coordinates {(0,-7) (1.3,-6.5) (2.5,-6) (4,-6.2) (5,-7)};
\draw [dashed] plot[smooth] coordinates {(-0.4,-5.6) (0.8,-5) (4,-5.6) (5.3,-5.3)};
\node at (2.7,-1.6) {$\F$};
\node at (2.5,-4.9) {$\D$};
\node at (-0.8,-5) {$\B$};
\node at (7.1,-1.8) {$\n_{\partial\F}$};
\node at (7.2,-4.8) {$\n_\B$};
\draw[->] (5,-1.8)--(6.3,-1.8);
\draw[->] (5.3,-5.3)--(6.6,-4.8);
\end{tikzpicture}
\caption{An illustration of the domain $\D$ with boundary $\partial\D = \bar\F\cup\B$ for the linearized problem in \eqref{eq:SloshingTension}. We assume $\n_{\partial\F}$ agrees with $\n_\B$ on $\partial\F$.}
\label{f:diagram}
\end{figure}

We further simplify \eqref{eq:Slosh1} as follows. Consider an equilibrium solution $(\phi_0,\eta_0)=(c,0)$ of \eqref{eq:Slosh1}, where $c$ is any constant scalar function (which gives zero velocity field). Assuming the free surface displacement $\eta$ is a small perturbation of $\{z=0\}$,  we look for solutions of the form $\phi(x,y,z,t) = c + \varepsilon\hat\phi(x,y,z,t)$, $\eta(x,y,t) = \varepsilon\hat\eta(x,y,t)$,  where $\varepsilon>0$ is some small parameter and collect  $\O(\varepsilon)$ terms. 
Next, we Taylor expand $\hat\phi$ and its derivatives around $z=0$. This transforms the boundary conditions,  \eqref{eq:Slosh1c} and \eqref{eq:Slosh1d}, from $\F_T$ to $\F$.  

Finally, time harmonic solutions (with angular frequency $\omega$ and phase shift $\delta$) are sought, via the ansatz 
\begin{align*} 
\hat\phi(x,y,z,t) &= \Phi(x,y,z)\cos(\omega t + \delta), \\ 
\hat\eta(x,y,t) &=\xi(x,y)\sin(\omega t + \delta), 
\end{align*}
where $\Phi(x,y,z)$ and $\xi(x,y)$ are the sloshing velocity potential and height respectively.  We  obtain the linearized eigenvalue problem for $(\omega,\Phi,\xi)$, which we refer to as the \emph{sloshing problem with surface tension}:
\begin{subequations}\label{eq:SloshingTension}
\begin{alignat}{2}
\label{eq:SloshingTension1} \Delta\Phi &=0 && \qquad\textrm{ in }\D, \\ 
\label{eq:SloshingTension2} \partial_\n\Phi &= 0 && \qquad\textrm{ on }\B, \\ 
\label{eq:SloshingTension3} \Phi_z &= \omega\xi && \qquad\textrm{ on }\F, \\ 
\label{eq:SloshingTension4} \xi - \frac{1}{\bond}\Delta_\F\xi  &= \omega\Phi && \qquad\textrm{ on }\F, \\
\label{eq:SloshingTension5} \partial_\n\xi &=0 && \qquad\textrm{ on }\partial\F.  
\end{alignat}
\end{subequations}
Here, $\nabla_\F\coloneqq (\partial_x,\partial_y)$, $\Delta_\F\coloneqq\nabla_\F\cdot\nabla_\F=\partial_{xx} + \partial_{yy}$, $\Phi_z=\partial_z\Phi$ and $\D$ is the fixed reference domain, with boundary $\partial\D=\F\cup\B$; see Figure \ref{f:diagram}. This problem must also be complemented with the condition $\int_\F\xi\, dA=0$, which amounts to mass conservation of the fluid. Since we are only interested in nontrivial solutions of \eqref{eq:SloshingTension}, we exclude the trivial solution $(\omega_0,\Phi_0,\xi_0)=(0,1,0)$ by imposing the orthogonality condition $\int_\F\Phi\, dA = 0$. Interestingly, the spectral parameter, $\omega$, appears in the boundary condition on the free surfaces, \eqref{eq:SloshingTension3} and \eqref{eq:SloshingTension4}.

\subsection{Zero suface tension.} \label{sec:ZeroSurfTen}
We summarize some well-known results for the case of zero surface tension corresponding to $\bond=\infty$, which has received considerable attention in the literature; see, for example,  \cite{Troesch:1960aa, Troesch:1965aa, Fox:1983aa, Kozlov:2004aa,  Ibrahim:2005aa,  Banuelos:2010aa, Kozlov:2011aa, Kulczycki:2014aa}. 

When $\bond = \infty$, we see from \eqref{eq:SloshingTension4} that the free surface height $\xi$ is proportional to the sloshing mode $\Phi$ restricted to the free surface $\F$ and can be eliminated from \eqref{eq:SloshingTension}. This yields the greatly simplified eigenvalue problem for $(\omega,\Phi)$
\begin{subequations}\label{eq:Steklov}
\begin{alignat}{2}
\Delta\Phi &=0 && \qquad\textrm{ in }\D, \\ 
\partial_\n\Phi &= 0 && \qquad\textrm{ on }\B, \\
\Phi_z & = \omega^2 \Phi && \qquad\textrm{ on }\F,  
\end{alignat}
\end{subequations}
which is commonly referred to as the \emph{mixed Steklov-Neumann eigenvalue problem} or the \emph{sloshing problem}. It is known \cite{Moiseev:1964aa, Kopachevsky:2012aa} that, if $\D$ and $\F$ are Lipschitz domains, then \eqref{eq:Steklov} has a discrete sequence of eigenvalues
\[ 0 = \omega_0^2 <\omega_1^2\le\omega_2^2\le\cdots \ \ \textrm{ with } \omega_n^2\longrightarrow\infty\textrm{ as }n\longrightarrow\infty. \]
The corresponding eigenfunctions $\{\Phi_n\}_{n=0}^\infty$ belong to the Sobolev space $H^1(\D)$ and when restricted to the free surface $\F$ form a complete orthogonal set in $L^2(\F)$. The eigenvalues, $\omega_n^2$, can be characterized by means of a variational principle \cite{Moiseev:1964aa,Troesch:1965aa}:
\begin{equation}\label{eq:SteklovVF}
\inf_{\Phi\in H_n} \ \   \int_\D |\nabla\Phi|^2\, dV  \qquad
\textrm{subject to} \ \  \|\Phi\|_{L^2(\F)} =1, 
\end{equation}
where $H_n$ is  defined by
$$ H_n = \left\{\Phi\in H^1(\D)\colon \int_\F\Phi\Phi_j\, dA = 0\textrm{ for all }j=0,1,\ldots,n-1\right\}, $$
where $\Phi_j$ is the $j$-th eigenfunction of \eqref{eq:Steklov}. 
Here $\Phi_0$ is the constant solution corresponding to $\omega_0=0$.

It is worth mentioning that the fundamental eigenfunction $\Phi_1$ corresponding to the fundamental (first nontrivial) eigenvalue $\omega_1^2$ can be used to determined the ``high spot", the maximal elevation of the free surface of the sloshing fluid. See, for example, \cite{Lamb:1932aa} for such relation. Several results about the location of high spots for different container geometries in two and three dimensions were obtained in \cite{Kulczycki:2009aa, Kulczycki:2011aa, Kulczycki:2012aa}. Moreover, it was shown in \cite{Kulczycki:2009aa} that for vertical-walled containers with constant depth, the question about high spots is equivalent to the \emph{hot spots conjecture} formulated by Rauch. See \cite{Burdzy:2013aa} for a recent review.

\subsection{Main results.}
Modeling irrotational water waves using variational principles has been investigated recently in \cite{Clamond:2012aa}. There are mainly two variational principles: the Hamiltonian of Petrov-Zakharov \cite{Petrov:1964aa, Zakharov:1968aa} and the Lagrangian of Luke \cite{Whitham:1965aa, Luke:1967aa, Whitham:1967aa}. In this paper, we derive a variational principle similar to Luke, in the sense that it is of free boundary type \cite[pp 208]{Courant:1953aa}.  Let $\H$ be the direct sum of Sobolev spaces defined by
\[ \H = \left\{ (\Phi,\xi)\in H^1(\D)\times H^1(\F)\colon \int_\F\Phi\, dA = 0 = \int_\F\xi\, dA\right\}. \]
Define the \emph{Dirichlet energy} of $\Phi\in H^1(\D)$ and the \emph{free surface energy} of $\xi\in H^1(\F)$ by
\[ D[\Phi] = \frac{1}{2}\int_\D |\nabla\Phi|^2\, dV \qquad\text{and}\qquad  S[\xi] = \frac{1}{2}\int_\F\left(\xi^2 + \frac{1}{\bond}|\nabla_\F\xi|^2\right)\, dA, \]
respectively. Our main result is the following theorem giving a variational characterization of the fundamental   eigenvalue of \eqref{eq:SloshingTension}. 
\begin{theorem} \label{thm:VFsum_firstmode}
There exists a minimizer $(\Phi_1,\xi_1)$ to the following minimization problem: 
\begin{equation}\label{eq:VFsum_firstmode}
\inf_{(\Phi,\xi)\in\H} \ \   D[\Phi] + S[\xi]  \quad
\textrm{subject to} \ \  \langle \Phi,\xi \rangle_{L^2(\F)} =1 .
\end{equation}
Moreover, $(\Phi_1,\xi_1)$ is an eigenfunction of \eqref{eq:SloshingTension} in the weak sense with corresponding eigenvalue $\omega_1 = D[\Phi_1] + S[\xi_1]$.
\end{theorem}
We also prove a Rayleigh-Ritz generalization of Theorem \ref{thm:VFsum_firstmode} for higher eigenvalues; see Theorem \ref{thm:VFsum_highmode}. An interesting feature of both variational characterizations are the constraints involving the $L^2$ inner product on the free surface $\F$, requiring the sloshing mode and the free surface height to have unit inner product and be orthogonal to lower modes; see Lemma \ref{thm:Property1}. 

\begin{remark}It is not difficult to show that if $(\phi(x,y,z,t),\xi(x,y,t))$ satisfies the time-dependent linear sloshing problem \eqref{eq:Slosh2}, then the quantity $E(t) = D[\phi(t)] + S[\eta(t)]$ is conserved.  
\end{remark}

In Theorem \ref{thm:domain}, we prove a domain monotonicity result, analogous to a result in \cite{Moiseev:1964aa},  for the fundamental eigenvalue of \eqref{eq:SloshingTension}. In Section \ref{sec:CompKopachevsky}, we describe the variational formulation for the sloshing problem \eqref{eq:SloshingTension} of Kopachevsky and Krein \cite{Kopachevsky:2012aa} and compare to the present work. 

In Corollary \ref{cor:ZeroSurfaceTension}, we establish that in the limit of zero surface tension, ($\bond = \infty$), the variational principle in Theorem \ref{thm:VFsum_firstmode} reduces to the mixed Steklov-Neumann variational principle \eqref{eq:SteklovVF}.  In Theorem \ref{thm:AsympFirstOrder}, we give the first-order perturbation formula for a simple eigenvalue satisfying  \eqref{eq:SloshingTension} in the limit where the Bond number is large. Finally, we illustrate Theorem \ref{thm:AsympFirstOrder} with a cylindrical container, where the exact solution is known.

\subsection{Outline.} This paper is structured as follows. We begin by discussing the contact angle and its role in contact line boundary condition \eqref{eq:Slosh1e} in Section \ref{sec:ContactAngle}. 
In Section \ref{sec:Prelim}, we prove preparatory results for  Theorem \ref{thm:VFsum_firstmode}. 
We  prove Theorem \ref{thm:VFsum_firstmode} in Section \ref{sec:Proof} and provide a Rayleigh-Ritz generalization of Theorem \ref{thm:VFsum_firstmode} for higher eigenvalues. Section \ref{sec:Large Bond Number Asymptotics} describes the asymptotic behavior of the eigenvalue $\omega$ in the limit where $\bond$ is large. We conclude in Section \ref{sec:disc} with a discussion. In Appendix \ref{sec:AppA}, we give a physical derivation of the sloshing problem with surface tension, \eqref{eq:SloshingTension}.

\section{Contact angle and its relation with contact line boundary condition} \label{sec:ContactAngle}
It can be seen in Appendix \ref{sec:AppA} that including surface tension effects on the free surface $\F_T$ introduces additional terms involving second derivatives of $\eta$ onto the dynamic boundary condition \eqref{eq:Slosh1d} on $\F_T$. It is thus deemed necessary to impose a boundary condition on $\partial\F_T$ so that the sloshing problem \eqref{eq:SloshingTension} is well-posed. Such a boundary condition, commonly referred to as the \emph{contact-line boundary condition}, controls the free surface height at the contact point, {\it{i.e.}} the point at which the contact line intersects the container's wall \cite{Hocking:1987aa}. 

The contact angle, defined in Subsection \ref{sec:SurfTenEff}, plays an important role in describing the contact line behavior. As first described by Young in his celebrated essay \cite{Young:1805aa}, the \emph{static contact angle} $\theta_s$ (also called Young's angle) is characterized by the following equation
$ T_{LG}\cos\theta_s = T_{SG} - T_{SL}, $ 
where $T_{LG}, T_{SG}, T_{SL}$ represents the liquid-gas, solid-gas, and solid-liquid surface tension, respectively. Once the contact line is in motion, one should expect the contact angle to be different from $\theta_s$; such contact angle is then called the \emph{dynamic contact angle} $\theta_d$. Accordingly, the static contact angle should remain unchanged in static conditions; however, experimental evidence  \cite{Dussan:1979aa, Cocciaro:1991aa, Cocciaro:1993aa}  demonstrates that this is false in general. In fact, the static contact angle lies between a range $\theta_R\le \theta_s\le \theta_A$, where $\theta_R$ and $\theta_A$ are the so-called receding and advancing contact angle respectively. Such  behavior is known as the \emph{contact angle hysteresis}, and surface roughness and/or heterogeneity of the container wall seem to be the reason behind this.

It is therefore extremely difficult to derive boundary conditions that takes into account both the contact angle hysteresis and the dynamic behavior of the contact line. We list three contact-line boundary conditions proposed in the study of capillary-gravity waves, each of which works under different assumptions. See \cite{Pomeau:2002aa} for a recent review.

\begin{enumerate}
\item \emph{Free-end edge constraint} (Neumann-type), which has the form 
$\partial_\n\eta = 0$ on $\partial\F_T$, where $\n$ is the normal to the solid boundary drawn into the fluid. This is a standard approach in studying capillary-gravity waves. This occurs if one assumes that the contact line can freely slip across the container's wall and $\theta_d\approx\theta_s$. Reynolds and Satterlee consider such a special case in \cite{Reynolds:1966aa}.

\item \emph{Pinned-end edge constraint} (Dirichlet-type), which has the form $\eta_t=0$ on $\partial\F_T$. This corresponds to fixing the contact line at the contact point (hence the word pinned) and assuming the dynamic contact angle $\theta_d$ lies within the interval $(\theta_R, \theta_A)$. This was first suggested by Benjamin and Scott \cite{Benjamin:1979aa} and investigated in \cite{Graham:1983aa, Graham:1984aa, Benjamin:1985aa, Henderson:1994aa}; however, these are all restricted to flat static interface or $\theta_s=\pi/2$. The case of curved static interface or $\theta_s\neq\pi/2$ was recently investigated by \cite{Shankar:2007aa}. It is worth mentioning that while this boundary condition makes the theoretical analysis much more difficult but still possible, it is not compatible with the kinematic condition at the container's wall \cite{Shankar:2005aa}. 

\item \emph{Wetting boundary condition} (Robin-type), which has the form
$ \eta_t = \lambda\partial_\n\eta$ on $\partial\F_T$, where $\lambda$ is some constant measuring the ratio of the contact line velocity to the change in contact angle. Observe that this model includes, as limiting cases, both the free-end $(\lambda=\infty)$ and the pinned-end $(\lambda=0)$ edge conditions. This was first proposed by Hocking \cite{Hocking:1987aa, Hocking:1987bb} and investigated by Miles  \cite{Miles:1990aa, Miles:1991aa, Miles:1992aa, Miles:1996aa} and Shen and Yeh \cite{Shen:1999aa}. The assumptions needed here are that the contact angle hysteresis $\theta_A-\theta_R$ is negligibly small, $\theta_s=\pi/2$, and $\theta_d$ is an linear function of the contact line velocity. 
\end{enumerate}

In this paper, we assume that the static contact angle is $\theta_s=\pi /2$ and the contact angle hysteresis is negligibly small; this is physically achieved by a container with smooth walls and a fluid that is free of contamination.  It can then be shown  \cite{Shankar:2005aa} that the contact angle remains unchanged, {\it{i.e.}} $\theta_d=\pi/2$. Assuming that the contact line slips freely, we can write down the boundary condition \eqref{eq:Slosh1e}
\[ 0 = \cos(\theta_d) = -\n_\B\cdot\n_{\F_T}\qquad\textrm{ on }\partial\F_T. \]

Another consequence of this assumption is that the static meniscus $\F$ is flat everywhere. Assuming constant surface tension $T_{LG} = T$, its shape, which we denote by $S(\tilde x,\tilde y)$, is governed by the Young-Laplace equation \cite{Finn:1986aa, Bush:2013aa}:
\begin{equation}
\label{eq:YoungLaplace} \rho gS = -T\nabla\cdot\n_\F. 
\end{equation}
Since $\n_\B = \n_{\partial\F}$ on $\partial\F$ and $\theta_s=\pi/2$, the contact line boundary condition becomes $\partial_\n S=0$ on $\partial\F$. Next, assuming $S_{\tilde x\tilde x}^2 + S_{\tilde y\tilde y}^2\ll 1$ (small slope approximation), we can linearize the Young-Laplace equation; upon nondimensionalizing the system, we obtain the dimensionless linearized Young-Laplace equation
\begin{subequations} \label{eq:NDYoungLaplace}
\begin{alignat}{2}
s_{xx} + s_{yy} = \Delta_\F s & = (\bond)s && \qquad\textrm{ in }\F, \\
\partial_\n s & = 0 && \qquad\textrm{ on }\partial\F.
\end{alignat}
\end{subequations}
The trivial solution $s(x,y)\equiv 0$ exists for problem \eqref{eq:NDYoungLaplace} but since $\bond$ is assumed to be positive, an energy argument shows that there is no nontrivial solution.


\section{Preliminary results} \label{sec:Prelim}
In this section we collect a range of auxiliary results that are required in the proof of Theorems \ref{thm:VFsum_firstmode} and \ref{thm:VFsum_highmode}. 

\subsection{Properties of Solutions to \eqref{eq:SloshingTension}.} \label{sec:Prelim_Properties}
\begin{theorem} \label{thm:Property1}
Suppose $(\omega,\Phi,\xi), (\omega_j,\Phi_j,\xi_j), (\omega_k,\Phi_k,\xi_k)$ are weak solutions of \eqref{eq:SloshingTension}. 
\begin{itemize}
\item[(a)] If $\omega\neq 0$, then $\langle\Phi,1\rangle_{L^2(\F)} = 0 = \langle\xi,1\rangle_{L^2(\F)}$.
\item[(b)] We have the identities 
\begin{subequations}\label{eq:P1swap}
\begin{align}
\label{eq:P1swap1} \int_\F (\Phi_j)_z\Phi_k\, dA & = \int_\F \Phi_j(\Phi_k)_z\, dA, \\
\label{eq:P1swap2} \int_\F (\Delta_\F\xi_j)\xi_k\, dA & = \int_\F \xi_j(\Delta_\F\xi_k)\, dA.
\end{align}
\end{subequations}
\item[(c)] If $|\omega_j|\neq |\omega_k|$, the following orthogonality condition holds: 
\begin{equation}
\label{eq:P2ortho} \langle\Phi_j, \xi_k\rangle_{L^2(\F)} = 0 = \langle\xi_j, \Phi_k\rangle_{L^2(\F)}. 
\end{equation}
\end{itemize}
\end{theorem}

\begin{proof}
Part (a) is obtained by simply integrating \eqref{eq:SloshingTension} over respective domains and applying divergence theorem. Part (b) is an easy consequence of the divergence theorem. We now prove part (c) using part (b). First, substituting \eqref{eq:SloshingTension3} for both $\Phi_j, \Phi_k$ into \eqref{eq:P1swap1} yields
\[ \omega_j\int_\F\xi_j\Phi_k\, dA = \int_\F(\Phi_j)_z\Phi_k\, dA = \int_\F\Phi_j(\Phi_k)_z\, dA= \omega_k\int_\F\Phi_j\xi_k\, dA. \]
Similarly, substituting \eqref{eq:SloshingTension4} for both $\xi_j,\xi_k$ into \eqref{eq:P1swap2} yields 
\[ \int_\F(-\omega_j\Phi_j + \xi_j)\xi_k\, dA = \int_\F\left(\frac{1}{\bond}\Delta_\F\xi_j\right)\xi_k\, dA = \int_\F\xi_j\left(\frac{1}{\bond}\Delta_\F\xi_k\right)\, dA = \int_\F\xi_j(-\omega_k\Phi_k + \xi_k)\, dA. \]
Rearranging these equations gives
\begin{subequations}
\begin{align}
\label{eq:P2ortho1} \omega_j\int_\F\xi_j\Phi_k\, dA - \omega_k\int_\F\Phi_j\xi_k\, dA & = 0 \\
\omega_j\int_\F\Phi_j\xi_k\, dA - \omega_k\int_\F\xi_j\Phi_k\, dA & = 0, 
\end{align}
\end{subequations}
which can be written as a linear system
\[ A\begin{pmatrix} \omega_j \\ \omega_k \end{pmatrix} =  \begin{pmatrix} 0 \\ 0 \end{pmatrix}, \textrm{ where } 
A = \begin{bmatrix} \langle\xi_j, \Phi_k\rangle_{L^2(\F)} & -\langle\Phi_j, \xi_k\rangle_{L^2(\F)} \\ \langle\Phi_j, \xi_k\rangle_{L^2(\F)} & -\langle\xi_j, \Phi_k\rangle_{L^2(\F)} \end{bmatrix}. \]
A nontrivial solution exists for the linear system if and only if $\det(A)= 0$, {\it i.e.}
\begin{align*}
\langle\Phi_j, \xi_k\rangle_{L^2(\F)}^2 - \langle\xi_j, \Phi_k\rangle_{L^2(\F)}^2 = 0 & \implies \langle\Phi_j, \xi_k\rangle_{L^2(\F)} = \pm \langle\xi_j, \Phi_k\rangle_{L^2(\F)}. 
\end{align*}
But using \eqref{eq:P2ortho1} and  $|\omega_j|\neq |\omega_k|$, we obtain \eqref{eq:P2ortho}.
\end{proof}


\begin{lemma}\label{thm:Property4_frequency}
Suppose $(\Phi,\xi,\omega)$ is a weak solution of \eqref{eq:SloshingTension}. We have the following expression for $\omega$:
\begin{equation}\label{eq:formsum}
\omega = \dfrac{\int_\D|\nabla\Phi|^2\, dV + \int_\F \left(\xi^2 + \frac{1}{\bond}|\nabla_\F\xi|^2\right)\, dA}{2\int_\F \Phi\xi\, dA} = \frac{D[\Phi] + S[\xi]}{\int_\F \Phi\xi\, dA} .
\end{equation}
In particular, 
\begin{equation}\label{eq:energy_equidist}
D[\Phi] = S[\xi] = \frac{\omega}{2}\langle\Phi,\xi\rangle_{L^2(\F)}.
\end{equation}
\end{lemma}
\begin{proof}
First, integrating both \eqref{eq:SloshingTension2}, \eqref{eq:SloshingTension5} against $\Phi,\xi$ over $\B, \partial\F$, respectively, gives
\[ \int_\B \partial_\n\Phi\Phi\, dA = 0\qquad\textrm{ and }\qquad\int_{\partial\F} \partial_\n\xi\xi\, ds = 0. \]
Next, integrating \eqref{eq:SloshingTension1} against $\Phi$ over $\D$ and applying divergence theorem gives
\[ 0 = \int_\D (\Delta\Phi)\Phi\, dV = \int_{\B\cup\F} \partial_\n\Phi\Phi\, dA - \int_\D|\nabla\Phi|^2\, dV = \int_\F\Phi_z\Phi\, dA - \int_\D|\nabla\Phi|^2\, dV. \]
Integrating \eqref{eq:SloshingTension3} against $\Phi$ over $\F$ and using the equation above gives:
\begin{equation}
\label{eq:P4freq1} \omega\int_\F\Phi\xi\, dA = \int_\F\Phi_z\Phi\, dA = \int_\D|\nabla\Phi|^2\, dV.
\end{equation}
Next, integrating  \eqref{eq:SloshingTension4} against $\xi$ over $\F$ and applying the divergence theorem gives
\begin{equation}
\label{eq:P4freq2} \omega\int_\F\Phi\xi\, dA = \int_\F\xi^2\, dA - \dfrac{1}{\bond}\int_\F (\Delta_\F\xi)\xi\, dA = \int_\F\xi^2\, dA + \dfrac{1}{\bond}\int_\F|\nabla_\F\xi|^2\, dA.
\end{equation}
The result follows from summing \eqref{eq:P4freq1}, \eqref{eq:P4freq2} and rearranging terms.
\end{proof}

\subsection{Direct Method from the Calculus of Variations.} \label{sec:Prelim_CoV}
This subsection establishes results for the functional in \eqref{eq:VFsum_firstmode} so that we may apply the direct method from the calculus of variations \cite{Dacorogna:2007aa, Evans:2010aa} to prove Theorem \ref{thm:VFsum_firstmode}. We begin by reminding the reader that $\D\subset\R^3$ is assumed to be a bounded Lipschitz domain, and the Sobolev space $H^1(\D)$ admits a natural inner product, given by
$ \langle v,w\rangle_{H^1(\D)} = \langle v,w\rangle_{L^2(\D)} + \langle\nabla v, \nabla w\rangle_{L^2(\D)}$  for any $v, w\in H^1(\D)$
with  induced norm
$ \|v\|_{H^1(\D)}^2 = \|v\|_{L^2(\D)}^2 + \|\nabla v\|_{L^2(\D)}^2$. 
 For any $\Phi\in H^1(\D)$ and $\xi\in H^1(\F)$, we denote by $[\Phi]_\F, [\xi]_\F$ the average value (mean) of $\Phi,\xi$ over $\F$, respectively. That is,
\[ [\Phi]_\F = \frac{1}{|F|}\int_\F \Phi\, dA  \quad  \textrm{and} \quad [\xi]_\F = \frac{1}{|F|}\int_\F \xi\, dA, \]
where $|F|$ denotes the two-dimensional Lebesgue measure of $\F$; here $[\Phi]_\F$ is understood in the sense of trace \cite[Chapter 5.5]{Evans:2010aa}. The first result shows that the space of functions in Theorem \ref{thm:VFsum_firstmode} is a Hilbert space.

\begin{lemma}\label{thm:CoV_DirectSum}
The space of functions
\[ \H = \left\{(\Phi,\xi)\in H^1(\D)\times H^1(\F)\colon \int_\F \Phi\, dA = 0 = \int_\F \xi\, dA\right\} \]
is a Hilbert space with its induced norm
$ \|(\Phi,\xi)\|_\H^2 = \|\Phi\|_{H^1(\D)}^2 + \|\xi\|_{H^1(\F)}^2. $
\end{lemma}
\begin{proof}
Define the following function spaces: 
\begin{align*}
X_\D & = \left\{\Phi\in H^1(\D)\colon \int_\F \Phi\, dA=0\right\} 
\quad \textrm{and} \quad 
X_\F  = \left\{\xi\in H^1(\F)\colon \int_\F \xi\, dA = 0\right\}.
\end{align*}
We first show that $X_\D, X_\F$ are closed subspaces of $H^1(\D), H^1(\F)$, respectively. It is clear that both $X_\D, X_\F$ are subspaces. Consider any $\Phi\in\bar X_\D$, the closure of $X_\D$. There exists a sequence $(\Phi_j)\in X_\D$ such that $\Phi_j\longrightarrow\Phi$ in $H^1(\D)$. Using the continuity of the trace operator $\Gamma_\D\colon H^1(\D)\longrightarrow L^2(\partial\D)$ \cite{Evans:2010aa},
\begin{align*}
\left|\int_\F\Phi\, dA - \int_\F\Phi_j\, dA\right|& \le \int_\F |\Phi-\Phi_j|\, dA  \le |F|^{1/2}\|\Phi-\Phi_j\|_{L^2(\F)}\\
& \le |F|^{1/2}\|\Phi-\Phi_j\|_{L^2(\partial\D)} = |F|^{1/2}\|\Gamma_\D(\Phi-\Phi_j)\|_{L^2(\partial\D)}\\
& \le C_\Gamma |F|^{1/2}\|\Phi-\Phi_j\|_{H^1(\D)}\longrightarrow 0\textrm{ as }j\longrightarrow\infty.
\end{align*}
This shows that $X_\D$ is closed in $H^1(\D)$ since
\[ 0=\int_\F\Phi_j\, dA \longrightarrow \int_\F\Phi\, dA. \]
A similar argument using only the Cauchy-Schwarz inequality shows that $X_\F$ is closed in $H^1(\F)$. Finally, since the closed subspace of a Hilbert space is also a Hilbert space, the direct sum of $X_\D$ and $X_\F$, which is $\H$, is a Hilbert space, with its inner product defined by
$ \langle v_1,v_2\rangle_\H = \langle\Phi_1,\Phi_2\rangle_{X_\D} + \langle\xi_1,\xi_2\rangle_{X_\F}$
with $v_1=(\Phi_1,\xi_1), v_2=(\Phi_2,\xi_2)\in\H$.
\end{proof}

To apply the direct method, one needs to verify that $D[\Phi]+S[\xi]$ satisfy coercivity and sequentially weakly lower-semicontinuity over $\H$; the latter follows since both $D[\Phi], S[\xi]$ possess some convexity property, which we will make precise in Lemma \ref{thm:CoV_Coercivity}. Coercivity means that $D[\Phi]+S[\xi]$ admits some lower growth condition in terms of $\|(\Phi,\xi)\|_{\H}^2$. The structure of $\H$ clearly suggests inequality of the Wirtinger type to estimate $D[\Phi]+S[\xi]$. As $\Phi$ has zero mean over $\F\subset\partial D$ instead of $\D$, a variant of the classical Poincar\'e-Wirtinger inequality, stated below, is applicable in showing coercivity, as we shall see in Lemma \ref{thm:CoV_Coercivity}.

\begin{theorem}[{\cite[Example 3.6]{Alessandrini:2008aa}}]\label{thm:CoV_Poincare_Wirtinger}
Consider a bounded Lipschitz domain $\Omega\subset\R^n, n\ge 1$. Let $\Gamma_\Omega\colon H^1(\Omega)\longrightarrow L^2(\partial\Omega)$ be the trace operator. For any open portion $\Sigma\subset\partial\Omega$, the following inequality holds for any $v\in H^1(\Omega)$:
\[ \|v - [v]_\Sigma\|_{H^1(\Omega)}\le \left(1+C_{\Gamma_\Omega}\left(\dfrac{|\Omega|}{|\Sigma|}\right)^{1/2}\right)\left(\sqrt{1+C_p^2}\right)\|\nabla u\|_{L^2(\Omega)}, \]
where $|\Omega|$ and $|\Sigma|$ are the $n$ and $(n-1)$ Lebesgue measure of $\Omega$ and $\Sigma$, respectively, and $C_{\Gamma_\Omega}, C_p$ positive constants that depends only on $\Omega$.
\end{theorem}

\begin{lemma}\label{thm:CoV_Coercivity}
The integral functional $\FF(v)=D[\Phi] + S[\xi]$  is weakly lower semicontinuous in $\H$ and satisfies the coercivity condition
\[ \FF(v)\ge C\|v\|_\H^2 \qquad \textrm{ for all  }  \ v=(\Phi,\xi)\in\H  \]
for some constant $C=C(\bond,\D,\F)>0$.
\end{lemma}
\begin{proof}
Observe that integrands of both $D[\Phi]$ and $S[\xi]$ are convex with respect to $\nabla\Phi$ and $\nabla_{\F}\xi$, respectively. It follows that they are weakly lower-semicontinuous in $H^1(\D), H^1(\F)$, respectively \cite[Theorem 2.12]{Rindler:2015aa}. Thus, for any $(v_j)=(\Phi_j,\xi_j)\rightharpoonup (\Phi,\xi) = v$ in $H^1(\D)\times H^1(\F)$ we have
\[ \FF(v) = D[\Phi] + S[\xi] \le \liminf_{j\in\N} D[\Phi_j] + \liminf_{j\in\N} S[\xi_j] \le \liminf_{j\in\N} \Big[D[\Phi_j] + S[\xi_j]\Big] = \liminf_{j\in\N} \FF(v_j). \]
The result follows since $\H$ is a subspace of $H^1(\D)\times H^1(\F)$.

Since $[\Phi]_\F=0$, Theorem \ref{thm:CoV_Poincare_Wirtinger} yields
\[ \|\Phi\|_{H^1(\D)}^2\le C(\D,\F)\|\nabla\Phi\|_{L^2(\D)}^2 =  2C(\D,\F)D[\Phi]. \]
On the other hand,
\[ 2S[\xi] = \int_{\F} \left(\xi^2 + \frac{1}{\bond}|\nabla_\F\xi|^2\right)\, dA \ge \min\left\{1,\frac{1}{\bond}\right\}\|\xi\|_{H^1(\F)}^2. \]
It follows that
\begin{align*}
\FF(v)  \ge \dfrac{1}{2C(\D,\F)}\|\Phi\|_{H^1(\D)}^2 + \min\left\{\frac{1}{2},\frac{1}{2\bond}\right\}\|\xi\|_{H^1(\F)}^2  
\geq C(\bond, \D, \F)\|v\|_\H^2.
\end{align*}
\end{proof}

Having established coercivity and sequentially weakly lower-semicontinuity, we prove the final  ingredient, which essentially says that the minimizing sequence will ``preserve" the integral constraint in the variational problem \eqref{eq:VFsum_firstmode}. The main tool in the proof below is the compactness of the trace operator $\Gamma_\D\colon H^1(\D)\longrightarrow L^2(\partial\D)$ \cite[pp 103]{Necas:2011aa}.

\begin{lemma}\label{thm:CoV_WC}
The function $(\Phi,\xi)\mapsto\displaystyle\int_\F \Phi\xi\, dA$ is weakly continuous in $H^1(\D)\times H^1(\F)$. 
\end{lemma}
\begin{proof}
Consider any (non-renumbered) subsequence of a weakly convergent sequence $v_j=(\Phi_j,\xi_j)\rightharpoonup (\Phi,\xi)=v$ in $H^1(\D)\times H^1(\F)$. Equivalently, $(\Phi_j)\rightharpoonup\Phi$ in $H^1(\D)$ and $(\xi_j)\rightharpoonup\xi$ in $H^1(\F)$. First, the Rellich-Kondrachov theorem implies that there exists a subsubsequence $(\xi_{j_k})\in H^1(\F)$ such that $\xi_{j_k}\longrightarrow\xi$ strongly in $L^2(\F)$. Recall that, since $\Gamma_\D$  is a compact linear operator, it maps weakly convergent sequences into strongly convergent sequences. Thus, $\Gamma_\D(\Phi_j)\longrightarrow\Gamma_\D(\Phi)$ strongly in $L^2(\partial\D)$. For this subsubsequence $(\Phi_{j_k},\xi_{j_k})$, the Cauchy-Schwarz inequality gives
\begin{align*}
\left|\int_\F\Phi_{j_k}\xi_{j_k}\, dA - \int_\F\Phi\xi\, dA\right| & \le \int_\F|\Phi_{j_k}-\Phi||\xi_{j_k}|\, dA + \int_\F|\Phi||\xi_{j_k}-\xi|\, dA\\
& \le \|\Phi_{j_k}-\Phi\|_{L^2(\F)}\|\xi_{j_k}\|_{L^2(\F)} + \|\Phi\|_{L^2(\F)}\|\xi_{j_k}-\xi\|_{L^2(\F)}\\
& \le \|\Gamma_\D(\Phi_{j_k})-\Gamma_\D(\Phi)\|_{L^2(\partial\D)}  \| \xi_{j_k}\|_{L^2(\F)} 
+ \|\Phi\|_{L^2(\F)} \|\xi_{j_k}-\xi\|_{L^2(\F)} \\ 
& \longrightarrow 0\textrm{ as }k\longrightarrow\infty,
\end{align*}
where we  used the fact that $\|\xi_{j_k}\|_{L^2(\F)}$ is bounded since $(\xi_{j_k})$ is a convergent sequence in $L^2(\F)$. This shows that
\[ \int_\F\Phi_{j_k}\xi_{j_k}\, dA\longrightarrow \int_\F\Phi\xi\, dA\quad\textrm{ as }k\longrightarrow\infty. \]
Since this is true for any subsequence of $(\Phi_j,\xi_j)$, the result follows. 
\end{proof}


\section{Proof of Theorem \ref{thm:VFsum_firstmode}, a Rayleigh-Ritz generalization, and other results} \label{sec:Proof}
We are now ready to prove Theorem \ref{thm:VFsum_firstmode}. An immediate consequence is the domain monotonicity property for the fundamental eigenvalue of \eqref{eq:SloshingTension}. We also prove a variational characterization of the higher eigenvalues of \eqref{eq:SloshingTension}.

\begin{proof}[Proof of Theorem \ref{thm:VFsum_firstmode}]
We begin by establishing the existence of a minimizer of \eqref{eq:VFsum_firstmode}, using the direct method from the calculus of variations; see  \cite[Theorem 2.36]{Rindler:2015aa}. Let $M=\{(\Phi,\xi)\in\H\colon\langle\Phi,\xi\rangle_{L^2(\F)}=1\}$. Choose a minimizing sequence $v_j=(\Phi_j,\xi_j)\in M$ such that
\[ D[\Phi_j]+S[\xi_j]\longrightarrow\inf_{(\Phi,\xi)\in M} \Big(D[\Phi]+S[\xi]\Big) = \omega_1. \]
Since bounded sets in reflexive Banach spaces are sequentially weakly relatively compact, Lemmas \ref{thm:CoV_DirectSum} and \ref{thm:CoV_Coercivity} imply  the existence of a weakly convergent subsequence $v_{j_k}\rightharpoonup v_1=(\Phi_1,\xi_1)$ in $\H$. Lemma \ref{thm:CoV_WC} asserts that $v_1$ satisfies the constraint $\langle\Phi^*,\xi^*\rangle_{L^2(\F)}=1$ so that $v_1\in M$, while Lemma \ref{thm:CoV_Coercivity} gives
\[ \omega_1\le D[\Phi_1] + S[\xi_1]\le\liminf_{j\in\N} \Big(D[\Phi_{j_k}] + S[\xi_{j_k}]\Big) = \omega_1. \]
Hence, $D[\Phi_1]+S[\xi_1]=\omega_1$ and $(\Phi_1,\xi_1)$ is a minimizer of \eqref{eq:VFsum_firstmode}. 

Let $(\Phi_1,\xi_1)$ be a minimizer to the problem \eqref{eq:VFsum_firstmode}. \emph{The method of Lagrange multipliers} leads us to consider the functional $J(\Phi,\xi)$ defined by
\[ \int_\D |\nabla\Phi|^2\, dV + \int_\F \Big(\xi^2+\dfrac{1}{\bond}|\nabla_\F\xi|^2\Big)\, dA - \lambda_1\int_\F\Phi\xi\, dA - \gamma_1\int_\F\Phi\, dA - \gamma_2\int_\F\xi\, dA \]
with Lagrange multipliers $\lambda_1, \gamma_1, \gamma_2\in\R$. 
For  a minimizer $(\Phi_1,\xi_1)$, the first variation of $J(\Phi,\xi)$ in the direction of $(f,g)\in H^1(\D)\times H^1(\F)$ must be zero. 
A direct computation gives the Euler-Lagrange equations
\begin{subequations}
\begin{alignat}{3}
\label{eq:EL1} \Delta\Phi_1 & = 0 && \qquad\textrm{ in }\D, \\
\label{eq:EL2} \partial_\n\Phi_1 & = 0 && \qquad\textrm{ on }\B, \\
\label{eq:EL3} (\Phi_1)_z & = \lambda_1\xi_1 + \gamma_1 && \qquad\textrm{ on }\F, \\
\label{eq:EL4} \xi_1 - \dfrac{1}{\bond}\Delta_\F\xi_1 & = \lambda_1\Phi_1 + \gamma_2 && \qquad\textrm{ on }\F, \\
\label{eq:EL5} \partial_\n\xi_1 & = 0 && \qquad\textrm{ on }\partial\F.
\end{alignat}
\end{subequations}
Note that integrating \eqref{eq:EL3}, \eqref{eq:EL4} over $\F$, using $\int_\F\Phi_1\, dA=0=\int_\F\xi_1\, dA$ and the divergence theorem gives:
\begin{align*}
\gamma_1\int_\F\, dA & = \int_\F (\Phi_1)_z\, dA - \lambda_1\int_\F\xi_1\, dA = \int_\D\Delta\Phi_1\, dV - \int_\B\partial_\n\Phi_1\, dA = 0, \\
\gamma_2\int_\F\, dA & = \int_\F\xi_1\, dA - \dfrac{1}{\bond}\int_\F\Delta_\F\xi_1\, dA - \lambda_1\int_\F\Phi_1\, dA = -\dfrac{1}{\bond}\int_{\partial\F}\partial_\n\xi_1\, ds = 0 . 
\end{align*}
Since $\displaystyle\int_\F\, dA\neq 0$, we must have $\gamma_1=\gamma_2=0$ and \eqref{eq:EL3}, \eqref{eq:EL4} reduce to
\begin{subequations}
\begin{alignat}{3}
(\Phi_1)_z & = \lambda_1\xi_1  && \quad\text{ on }\F, \label{eq:EL3a} \\
\xi_1 - \dfrac{1}{\bond}\Delta_\F\xi_1 & = \lambda_1\Phi_1 && \quad\text{ on }\F. \label{eq:EL4a}
\end{alignat}
\end{subequations}
Now, integrating \eqref{eq:EL3a}, \eqref{eq:EL4a} against $\Phi_1, \xi_1$, respectively, over $\F$ and using $\langle \Phi_1,\xi_1\rangle_{L^2(\F)}=1$ yields
\begin{align*}
\lambda_1 = \lambda_1\int_\F \Phi_1\xi_1\, dA & = \int_\F (\Phi_1)_z\Phi_1\, dA\\
& = \int_\D (\Delta\Phi_1)\Phi_1\, dV - \int_\B \partial_\n\Phi_1\Phi_1\, dA + \int_\D |\nabla\Phi_1|^2\, dV = \int_\D |\nabla\Phi_1|^2\, dV \\
\lambda_1 = \lambda_1\int_\F \Phi_1\xi_1\, dA & = \int_\F \xi_1^2\, dA -\dfrac{1}{\bond}\int_\F (\Delta_\F\xi_1)\xi_1\, dA\\
& = \int_\F \xi_1^2 dA - \dfrac{1}{\bond}\int_{\partial\F} \partial_\n\xi_1\, ds + \dfrac{1}{\bond}\int_\F |\nabla_\F\xi_1|^2\, dA = \int_\F \Big(\xi_1^2 + \dfrac{1}{\bond}|\nabla_\F\xi_1|^2\Big)\, dA.
\end{align*}
Finally, summing these two equations gives
\[ \lambda_1 = \dfrac{1}{2}\left\{\int_\D |\nabla\Phi_1|^2\, dV + \int_\F \Big(\xi_1^2 + \dfrac{1}{\bond}|\nabla_\F\xi_1|^2\Big)\, dA\right\} = D[\Phi_1] + S[\xi_1] = \omega_1. \]
\end{proof}

\begin{corollary}\label{thm:VFsum_equivalent}
The variational formulation \eqref{eq:VFsum_firstmode} is equivalent to 
\begin{equation}
\inf_{(\Phi,\xi)\in\H \setminus \{ 0 \}} \dfrac{D[\Phi] + S[\xi]}{ | \langle\Phi,\xi\rangle_{L^2(\F)} | }. \label{eq:VFsum_firstmode_equiv}
\end{equation}
\end{corollary}
\begin{proof}
Write $a = \langle\Phi,\xi\rangle_{L^2(\F)}$, which, without loss of generality, we may assume to be positive. Set $(\tilde\Phi, \tilde\xi) = ( \Phi / \sqrt a, \xi / \sqrt a ) $, where $(\Phi,\xi)\neq (0,0)$. Then $\langle\tilde\Phi, \tilde\xi\rangle_{L^2(\F)} = 1$ and
\begin{align*} 
 \inf_{(\Phi,\xi)\in\H\setminus\{ 0\}} \dfrac{D[\Phi] + S[\xi]}{a} 
 = \inf_{\substack{(\tilde\Phi,\tilde\xi)\in\H\\ \langle\tilde\Phi,\tilde\xi\rangle_{L^2(\F)}=1}} D[\tilde\Phi] + S[\tilde\xi].
\end{align*}
\end{proof}

In the following theorem, we prove a  domain monotonicity result about the fundamental eigenvalue of \eqref{eq:SloshingTension}, stating that if two containers have an identical free surface and both container walls are vertical at the free surface, then the larger container has a higher fundamental sloshing frequency. A similar result for the mixed Steklov-Neumann problem is given in  \cite{Moiseev:1964aa}. 

\begin{theorem}\label{thm:domain}
Suppose we have two bounded Lipschitz domains $\D,\tilde\D$ such that $\tilde\D\subset\D$ and the container's wall over which the contact line moves is vertical for both $\D,\tilde\D$. Suppose $\partial\D=\F\cup\B, \partial\tilde\D=\F\cup\tilde\B$, and $\B, \tilde\B$ are such that $\B$ envelops $\tilde\B$. Then $\omega_1^{\tilde\D} \leq \omega_1^\D$, where $\omega_1(\cdot)$ is the first non-trivial (positive) eigenvalue of \eqref{eq:SloshingTension}.
\end{theorem}
\begin{proof}
Denote by $D_\Omega[\Phi]$ the Dirichlet energy of $\Phi\in H^1(\Omega)$, where the domain of integration is $\Omega$. Since $\tilde\D\subset\D$, any function $\Phi\in H^1(\D)$ satisfies $D_{\tilde\D}[\Phi]\le D_\D[\Phi]$. Let $(\Phi,\xi), (\tilde\Phi,\tilde\xi)$ be minimizers of the variational problem \eqref{eq:VFsum_firstmode} over domains $\D,\tilde \D$, respectively, with corresponding minimum $\omega_1^\D, \omega_1^{\tilde\D}$. It follows that
\[ \omega_1^{\tilde\D} = \D_{\tilde\D}[\tilde\Phi] + S[\tilde\xi] \le D_{\tilde\D}[\Phi] + S[\xi] \leq D_\D[\Phi] + S[\xi] = \omega_1^\D. \]
\end{proof}

The variational formulation \eqref{eq:VFsum_firstmode} in Theorem \ref{eq:VFsum_firstmode} admits a Rayleigh-Ritz generalization for higher eigenvalues of \eqref{eq:SloshingTension}. 

\begin{theorem}\label{thm:VFsum_highmode}
For any fixed integer $m>1$, let $(\Phi_j,\xi_j)$, $j=1,\ldots,m-1$ be the first $m-1$ eigenfunctions of \eqref{eq:SloshingTension}. Define 
\begin{align*}
\H_m = \Big\{(\Phi,\xi)\in\H\colon & \langle\Phi,\xi_j\rangle_{L^2(\F)}= 0 =  \langle\xi,\Phi_j\rangle_{L^2(\F)}, \ \  j=1,\ldots,m-1 \Big\}.
\end{align*}
Consider the following minimization problem
\begin{align}\label{eq:VFsum_highmode}
\inf_{(\Phi,\xi)\in \H_m } \ \  D[\Phi] + S[\xi]
\qquad  
\textrm{subject to} \ \  \langle\Phi,\xi\rangle_{L^2(\F)} = 1. 
\end{align}
There exists a minimizer $(\Phi_m,\xi_m)$ to the problem \eqref{eq:VFsum_highmode}. Moreover, $(\Phi_m,\xi_m)$  is an eigenfunction of \eqref{eq:SloshingTension} in the weak sense with corresponding eigenvalue $\omega_m=D[\Phi_m]+S[\xi_m]$.
\end{theorem}
\begin{proof}
Observe that a similar argument in Lemma \ref{thm:CoV_DirectSum} shows that $\H_m$ is a Hilbert space. Moreover, Lemmas \ref{thm:CoV_Coercivity}, \ref{thm:CoV_Coercivity}, \ref{thm:CoV_WC} hold in $\H_m$ since it is a subspace of $\H$. Consequently, there exists a minimizer $(\Phi_m, \xi_m)$ to the minimization problem \eqref{eq:VFsum_highmode}. Let $(\Phi_m,\xi_m)$ be a minimizer to problem \eqref{eq:VFsum_highmode}. \emph{The method of Lagrange multipliers} leads us to consider the following functional $J(\Phi,\xi)$ defined by
\begin{align*}
\int_\D |\nabla\Phi|^2\, dV + \int_\F \Big(\xi^2 + \dfrac{1}{\bond}|\nabla_\F\xi|^2\Big)\, dA & - \lambda_m\int_\F \Phi\xi\, dA - \gamma_1\int_\F\Phi\, dA - \gamma_2\int_\F\xi\, dA\\
& - \sum_{j=1}^{m-1} \left(\alpha_j\int_\F \Phi\xi_j\, dA\right) - \sum_{j=1}^{m-1} \left(\beta_j\int_\F \xi\Phi_j\, dA\right)
\end{align*}
with Lagrange multipliers $\lambda_m, \gamma_1, \gamma_2, (\alpha_j), (\beta_j)\in\R$. For  a minimizer $(\Phi_m,\xi_m)$, the first variation of $J(\Phi,\xi)$ in the direction of $(f,g)\in H^1(\D)\times H^1(\F)$ must be zero. 
A direct computation gives the Euler-Lagrange equations
\begin{subequations}
\begin{alignat}{3}
\label{eq:EL11} \Delta\Phi_m & = 0 && \qquad\textrm{ in }\D,  \\
\label{eq:EL12} \partial_\n\Phi_m & = 0 && \qquad\textrm{ on }\B, \\
\label{eq:EL13} (\Phi_m)_z & = \lambda_m\xi_m + \gamma_1 + \sum_{j=1}^{m-1} \alpha_j\xi_j && \qquad\textrm{ on }\F, \\
\label{eq:EL14} \xi_m - \dfrac{1}{\bond}\Delta_\F\xi_m & = \lambda_m\Phi_m + \gamma_2 + \sum_{j=1}^{m-1} \beta_j\Phi_j && \qquad\textrm{ on }\F, \\
\label{eq:EL15} \partial_\n\xi_m & = 0 && \qquad\textrm{ on }\partial\F.
\end{alignat}
\end{subequations}
A similar argument in the proof of Theorem \eqref{thm:VFsum_firstmode} shows that $\gamma_1=\gamma_2=0$. Observe that by integrating \eqref{eq:EL13}, \eqref{eq:EL14} against $\Phi_k$, $\xi_k$, respectively, for some $k=1,\ldots,m-1$ and using Lemma \ref{thm:Property1}, we obtain
\[
\int_\F (\Phi_m)_z\Phi_k\, dA  = \lambda_m {\int_\F\xi_m\Phi_k\, dA} + \sum_{j=1}^{m-1}\alpha_j\int_\F \xi_j\Phi_k\, dA  = \alpha_k\int_\F\xi_k\Phi_k\, dA  \]
which implies by Lemma  \ref{thm:Property1},
\[ \alpha_k\int_\F \xi_k\Phi_k\, dA = \int_\F (\Phi_m)_z\Phi_k\, dA= \int_\F \Phi_m(\Phi_k)_z\, dA  = \omega_k\int_\F \Phi_m\xi_k\, dA = 0.\]
Next, using Lemma \ref{thm:Property1}, we have
\[\int_\F \xi_m\xi_k\, dA - \dfrac{1}{\bond}\int_\F (\Delta_\F\xi_m)\xi_k\, dA  = \lambda_m {\int_\F\Phi_m\xi_k\, dA} + \sum_{j=1}^{m-1} \beta_j\int_\F\Phi_j\xi_k\, dA = \beta_k\int_\F \Phi_k\xi_k\, dA,\]
which implies by Lemma \ref{thm:Property1}
\[ \begin{split}\beta_k \int_\F\Phi_k\xi_k\, dA & = \int_\F\xi_m\xi_k\, dA - \dfrac{1}{\bond}\int_\F (\Delta_\F\xi_m)\xi_k\, dA = \int_\F\xi_m\xi_k\, dA - \dfrac{1}{\bond}\int_\F \xi_m(\Delta_\F\xi_k)\, dA \\
& = \int_\F\xi_m\xi_k\, dA - \int_\F \Big(\xi_m(\xi_k - \omega_k\Phi_k)\Big)\, dA= \omega_k\int_\F \xi_m\Phi_k\, dA = 0.\end{split}
\]
Since $\displaystyle\int_\F\Phi_k\xi_k\, dA\neq 0$ for every $k=1,\ldots,m-1$, we must have $\alpha_k=\beta_k=0$ for every $k=1,\ldots,m-1$ and \eqref{eq:EL13}, \eqref{eq:EL14} reduce to
\begin{subequations}
\begin{alignat}{3}
(\Phi_m)_z & = \lambda_m\xi_m && \quad\text{ on }\F, \label{eq:EL13a} \\
\xi_m - \dfrac{1}{\bond}\Delta_\F\xi_m & = \lambda_m\Phi_m && \quad\text{ on }\F.  \label{eq:EL14a}
\end{alignat}
\end{subequations}
Finally, a similar argument in the proof of Theorem \ref{thm:VFsum_firstmode} shows that 
\[ \lambda_m = \dfrac{1}{2}\left\{\int_\D |\nabla\Phi_m|^2\, dV + \int_\F \Big(\xi_m^2 + \dfrac{1}{\bond}|\nabla_\F\xi_m|^2\Big)\, dA\right\} = D[\Phi_m] + S[\xi_m] = \omega_m. \]
\end{proof}
An analogous statement as in Corollary \ref{thm:VFsum_equivalent} holds for the higher modes.

\subsection{Comparison to the variational formulation of the sloshing problem with surface tension of Kopachevsky and Krein} \label{sec:CompKopachevsky}
In \cite[pp.207]{Kopachevsky:2012aa}, a  variational formulation for the eigenvalues (sloshing frequencies) of the sloshing problem with surface tension is given. It is worth noting that the authors work in a more general setting. 
\begin{enumerate}
\item The static contact angle satisfies $\theta_s\neq\pi/2$, which means that $\F$ is a curved surface. Upon linearization, this introduces additional coupled terms in the kinematic boundary condition on $\F$. To compensate for this, a  curvilinear coordinate system is introduced.  
\item The dynamic contact angle $\theta_d$ is shown to remain unchanged, and the contact-line boundary condition on $\partial\F$ is of Robin-type, having the form $\partial_\n\eta = -\chi\eta$, where $\chi$ is a dimensionless constant depending on $\theta_s$ and curvature on $\partial\F$. 
\end{enumerate}
In the present work, for simplicity, we have assumed a contact angle of $\theta_s=\pi/2$ and used Cartesian coordinates. Moreover, we assume $\chi=0$ so that the Neumann boundary condition $\partial_\n\eta=0$ on $\partial\F$ is recovered. Below, we discuss the results in \cite{Kopachevsky:2012aa} in this  setting. 

While seeking  time-harmonic solutions for the sloshing problem with surface tension, the authors use the same ansatz as ours for the free surface height $\hat\eta$ but a slightly different one for the velocity potential $\hat\phi$. They choose $\hat\phi(x,y,z,t) = \omega\varphi(x,y,z)\cos(\omega t)$. The sloshing problem with surface tension takes the form
\begin{subequations}\label{eq:SloshingTension_Kopachevsky}
\begin{alignat}{2}
\label{eq:SloshingTension_Kopachevsky1} \Delta\varphi & = 0 \ \ && \textrm{ in }\D, \\
\label{eq:SloshingTension_Kopachevsky2} \partial_\n\varphi & = 0 \ \ && \textrm{ on }\B, \\
\label{eq:SloshingTension_Kopachevsky3} \varphi_z & = \xi \ \ && \textrm{ on }\F, \\
\label{eq:SloshingTension_Kopachevsky4} \xi - \frac{1}{\bond}\Delta_\F\xi & = \omega^2\varphi \ \ && \textrm{ on }\F, \\
\label{eq:SloshingTension_Kopachevsky5} \partial_\n\xi & = 0 \ \ && \textrm{ on } \partial\F.
\end{alignat}
\end{subequations}
One can show by integrating \eqref{eq:SloshingTension_Kopachevsky1} against $\varphi$ and using divergence theorem that
\begin{equation}
\int_{\F}\varphi\xi\, dA = \int_{\F}\varphi_z\varphi\, dA = \int_{\D}|\nabla\varphi|^2\, dV = 2D[\varphi]. \label{eq:SloshingTension_Kopachevsky6}
\end{equation}

The system \eqref{eq:SloshingTension_Kopachevsky} is studied as follows.  Define the following spaces of functions
\begin{align*}
L_{\F}^2(\F) & = \left\{\xi\in L^2(\F)\colon \int_{\F} \xi\, dA = 0\right\}, \\
H_{\F}^{1/2}(\F) & = \left\{\xi\in H^{1/2}(\F)\colon \int_{\F} \xi\, dA = 0\right\}, \\
H_{\F}^{-1/2}(\F) & \coloneqq\left(H_{\F}^{1/2}(\F)\right)^*, \textrm{ the dual space of }H_{\F}^{1/2}(\F).
\end{align*}
Define the \emph{Neumann-to-Dirichlet} operator $C\colon H_{\F}^{-1/2}(\F)\longrightarrow H^{1/2}_{\F}(\F)$ such that $\xi \mapsto\varphi|_\F$, where $\varphi\in H^1(\D)$ is the unique solution of the Neumann problem 
\begin{alignat*}{2}
\Delta\varphi &=0 && \ \ \textrm{ in }\D, \\ 
\partial_\n\varphi &= 0 && \ \ \textrm{ on }\B, \\ 
\varphi_z & = \xi  && \ \ \textrm{ on }\F.  
\end{alignat*}
Projecting \eqref{eq:SloshingTension_Kopachevsky4} onto the space $L_{\F}^2(\F)$ and viewing LHS of the projected equation as an operator $B$ acting on $L_{\F}^2(\F)$ together with \eqref{eq:SloshingTension_Kopachevsky5}, we obtain the generalized eigenvalue problem
\begin{equation}
\xi - \frac{1}{\bond}\Delta_\F\xi = B\xi = \omega^2C\xi, \ \ \xi\in L_{\F}^2(\F), \label{eq:Kopachevsky_Operator}
\end{equation}
where $C$ is restricted to $L_{\F}^2(\F)$. Physically, the operators $C$ and $B$ correspond to the kinetic energy and potential energy operator respectively. 
It is proved that the fundamental eigenvalue $\omega_1^2$ in \eqref{eq:Kopachevsky_Operator} has the  variational characterization, 
\begin{subequations}
 \label{eq:Kopachevsky_VF}
\begin{alignat}{2}
\label{eq:Kopachevsky_VFa}
\omega_1^2 = \inf_{\substack{\xi\in L_{\F}^2(\F) \\ \varphi\in H^1(\D)}}  \ & \frac{S[\xi]}{D[\varphi]}  \\ 
\label{eq:Kopachevsky_VFb}
\text{subject to } \ \Delta\varphi & = 0  && \ \ \textrm{ in }\D, \\ 
\label{eq:Kopachevsky_VFc}
\partial_\n\varphi & = 0 && \ \ \textrm{ on }\B,  \\ 
\label{eq:Kopachevsky_VFd}
\varphi_z & = \xi && \ \ \textrm{ on }\F, \\ 
\label{eq:Kopachevsky_VFe}
\int_\F\varphi\, dA & = 0.
\end{alignat}
\end{subequations}

Comparing the variational formulations \eqref{eq:Kopachevsky_VF} and \eqref{eq:VFsum_firstmode}, we make the following observations.  Both variational formulations share the  constraint that the velocity potential and surface height  have zero mean over $\F$. In  \eqref{eq:VFsum_firstmode}, $(\Phi,\xi)$ need only satisfy the single integral constraint $\langle\Phi,\xi\rangle_{L^2(\F)} = 1$. However, in \eqref{eq:Kopachevsky_VF} each $\varphi$ must satisfy the constraint that $C\xi=\varphi|_{\F}$, {\it{i.e.}} they are solutions of the Laplace problem with Neumann data on $\F$ equal to $\xi$.

We claim that \eqref{eq:Kopachevsky_VF} follows  from  \eqref{eq:VFsum_firstmode}. To see this, we use the equivalent formulation of \eqref{eq:VFsum_firstmode} from Corollary \ref{thm:VFsum_equivalent}. Suppose $(\Phi_1,\xi_1)$ is a minimizer of \eqref{eq:VFsum_firstmode_equiv} with
\[ \omega_1 = \frac{D[\Phi_1] + S[\xi_1]}{|\langle\Phi_1,\xi_1\rangle_{L^2(\F)}|}>0. \]
Writing $\omega_1\varphi_1=\Phi_1$ and using the fact that $(\varphi_1,\xi_1)$ satisfies \eqref{eq:SloshingTension_Kopachevsky6}, we have
\[ \omega_1 = \frac{D[\omega_1\varphi_1] + S[\xi_1]}{|\langle\omega_1\varphi_1,\xi_1\rangle_{L^2(\F)}|} = \frac{\omega_1^2\left(D[\varphi_1] + \frac{S[\xi_1]}{\omega_1^2}\right)}{\omega_1|\langle\varphi_1,\xi_1\rangle_{L^2(\F)}|} = \omega_1\left(\frac{D[\varphi_1] + \frac{S[\xi_1]}{\omega_1^2}}{2D[\varphi_1]}\right). \]
Rearranging yields
$ \omega_1^2 = \frac{S[\xi_1]}{D[\varphi_1]}$. After specifying $\xi_1$,  we can obtain  $\varphi_1$ by solving 
the Neumann problem \eqref{eq:Kopachevsky_VFb}, \eqref{eq:Kopachevsky_VFc}, \eqref{eq:Kopachevsky_VFd}. 
It follows that  $(\varphi_1,\xi_1)\in\H$ minimizes \eqref{eq:Kopachevsky_VF}. 
However, it is not obvious how to  deduce the unconstrained formulation  \eqref{eq:VFsum_firstmode} directly from the constrained formulation \eqref{eq:Kopachevsky_VF}.


\section{Asymptotics} \label{sec:Large Bond Number Asymptotics} In this section, we consider the asymptotic limit where the Bond number,  $\bond$, is large for the sloshing problem with surface tension \eqref{eq:SloshingTension}. We first show that in the limit $\bond\to\infty$, i.e. zero surface tension, we recover the variational characterization for the mixed Steklov-Neumann problem or sloshing problem \eqref{eq:Steklov}, as derived by Troesch \cite{Troesch:1965aa}.

\begin{corollary} \label{cor:ZeroSurfaceTension}
Suppose $(\Phi_1,\xi_1)$ is a minimizer of the variational problem \eqref{eq:VFsum_firstmode} with $\bond=\infty$. Then $\sqrt{\omega_1}\Phi_1$ with $\omega_1=D[\Phi_1]+S[\xi_1]$ is a minimizer of the variational principle for the mixed Steklov-Neumann eigenvalue problem  \eqref{eq:SteklovVF}. 
\end{corollary}
\begin{proof}
From Theorem \ref{thm:VFsum_firstmode},  we know that $(\Phi_1,\xi_1)$ satisfies  the constraint $\langle\Phi_1,\xi_1\rangle_{L^2(\F)}=1$ and the following equation in the weak sense
\begin{equation}
\xi_1 = \omega_1\Phi_1 \ \ \textrm{ on }\F \label{eq:Reduction1}
\end{equation}
with $\omega_1=D[\Phi_1]+S[\xi_1]>0$. Integrating \eqref{eq:Reduction1} against $\xi_1$ over $\F$, together with the constraint yields $S[\xi_1]=\omega_1/2$; this also implies $D[\Phi_1]=\omega_1/2$. Defining $\tilde\Phi=\sqrt{\omega_1}\Phi_1$, integrating \eqref{eq:Reduction1} against $\Phi_1$ over $\F$, and using the constraint again yields
\[ 1=\omega_1\langle\Phi_1,\Phi_1\rangle_{L^2(\F)} = \langle\tilde\Phi,\tilde\Phi\rangle_{L^2(\F)} \]
and
\[ \int_\D |\nabla\tilde\Phi|^2\, dV = \omega_1\int_\D|\nabla\Phi_1|^2\, dV = \omega_1(2D[\Phi_1]) = \omega_1^2. \]
\end{proof}

We now investigate the asymptotic behavior of the eigenvalues of \eqref{eq:SloshingTension} in the limit where the Bond number is large. Let $\varepsilon=\bond^{-1}$ and $\omega(\varepsilon)$ be any eigenvalue satisfying \eqref{eq:SloshingTension} for a fixed $\varepsilon$. The previous result shows that $\omega(0)$ is an eigenvalue to the mixed Steklov-Neumann problem \eqref{eq:Steklov}. The following result gives the first perturbation for a simple eigenvalue, $\omega(\varepsilon)$  for $\varepsilon\ll 1$, {\it{i.e.}} $\bond\gg 1$. 

\begin{theorem}\label{thm:AsympFirstOrder}
Assume $\n_{\partial\F}(x) = \n_\B(x)$ for all $x\in\partial\F$. If $\omega^0 = \omega(0)$ is a simple eigenvalue, then  the derivative of $\omega = \omega(\varepsilon )$, satisfying \eqref{eq:SloshingTension}, with respect to $\varepsilon$ is given by 
\[ \frac{d\omega}{d\varepsilon} \bigg|_{\varepsilon = 0} = \frac{\omega^0}{2}\left(\frac{\|\nabla_\F\Phi^0\|_{L^2(\F)}^2}{\|\Phi^0\|_{L^2(\F)}^2}\right), \]
where $(\omega^0,\Phi^0)$ satisfy the mixed Steklov-Neumann problem \eqref{eq:Steklov}.
\end{theorem}
\begin{proof}
Consider the expansion of $(\omega,\Phi,\xi)$ in the form
\begin{align*}
\omega  &= \omega^0  + \varepsilon\omega^1  + \o(\varepsilon), \\
\Phi  &= \Phi^0 + \varepsilon\Phi^1 + \o(\varepsilon), \\
\xi  &= \xi^0 + \varepsilon\xi^1 + \o(\varepsilon).
\end{align*}
Substituting these expansions into \eqref{eq:SloshingTension} and collecting $\O(1)$ terms yields
\begin{subequations}
\begin{alignat}{2}
\label{eq:Asymp1} \Delta\Phi^0 & = 0 && \qquad\textrm{ in }\D, \\ 
\label{eq:Asymp2} \partial_\n\Phi^0 & = 0 && \qquad\textrm{ on }\B, \\
\label{eq:Asymp3} (\Phi^0)_z & = \omega^0\xi^0 && \qquad\textrm{ on }\F, \\
\label{eq:Asymp4} \xi^0 & = \omega^0\Phi^0 && \qquad\textrm{ on }\F, \\
\label{eq:Asymp5} \partial_\n\xi^0 & = 0 && \qquad\textrm{ on }\partial\F,
\end{alignat}
\end{subequations}
while collecting $\O(\varepsilon)$ terms yields the following PDEs:
\begin{subequations}\label{eq:PertEqns}
\begin{alignat}{2}
\label{eq:Asymp6} \Delta\Phi^1 & = 0 && \qquad\textrm{ in }\D,  \\
\label{eq:Asymp7} \partial_\n\Phi^1 & = 0 && \qquad\textrm{ on }\B, \\
\label{eq:Asymp8} (\Phi^1)_z & = \omega^0\xi^1 + \omega^1\xi^0 && \qquad\textrm{ on }\F, \\
\label{eq:Asymp9} \xi^1 - \Delta_\F\xi^0 & = \omega^0\Phi^1 + \omega^1\Phi^0 && \qquad\textrm{ on }\F, \\
\label{eq:Asymp10} \partial_\n\xi^1 & = 0 && \qquad\textrm{ on }\partial\F.
\end{alignat}
\end{subequations}
Multiplying \eqref{eq:Asymp8}, \eqref{eq:Asymp9}, against $\Phi^0, \xi^0$, respectively, and integrating over $\F$ yields
\begin{align}
\label{eq:Asymp8a} \int_\F (\Phi^1)_z\Phi^0\, dA & = \omega^0\int_\F\xi^1\Phi^0\, dA + \omega^1\int_\F \xi^0\Phi^0\, dA \\
\label{eq:Asymp9a} \int_\F \Big(\xi^1\xi^0 - (\Delta_\F\xi^0)\xi^0\Big) \, dA & = \omega^0\int_\F \Phi^1\xi^0\, dA + \omega^1\int_\F \Phi^0\xi^0\, dA.
\end{align}
One can easily deduce using \eqref{eq:Asymp1}, \eqref{eq:Asymp2}, \eqref{eq:Asymp6}, \eqref{eq:Asymp7} that
\[ \int_\F(\Phi^1)_z\Phi^0\, dA = \int_\F\Phi^1(\Phi^0)_z\, dA. \]
Thus, using \eqref{eq:Asymp4}, \eqref{eq:Asymp3}, \eqref{eq:Asymp9a}, \eqref{eq:Asymp5} and \eqref{eq:Asymp4}, \eqref{eq:Asymp8a} reduces to
\begin{alignat*}{2}
\omega^1\int_\F\xi^0\Phi^0\, dA & = \int_\F(\Phi^1)_z\Phi^0\, dA - \omega^0\int_\F\xi^1\Phi^0\, dA = \int_\F\Phi^1(\Phi^0)_z\, dA - \int_\F\xi^1\xi^0\, dA \\
& = \omega^0\int_\F\Phi^1\xi^0\, dA - \int_\F\xi^1\xi^0\, dA  = -\int_\F (\Delta_\F\xi^0)\xi^0\, dA - \omega^1\int_\F \Phi^0\xi^0\, dA \\
& = \int_\F |\nabla_\F\xi^0|^2\, dA - \omega^1\int_\F \Phi^0\xi^0\, dA 
\end{alignat*}
which implies
\[ \omega^1   = \frac{1}{2}\left(\dfrac{\int_\F|\nabla_\F\xi^0|^2\, dA}{\int_\F\Phi^0\xi^0\, dA}\right) = \frac{\omega^0}{2}\left(\dfrac{\int_\F|\nabla_\F\Phi^0|^2\, dA}{\int_\F |\Phi^0|^2\, dA}\right). \]
\end{proof}

\subsection*{Example: A cylindrical container}
We illustrate Theorem \ref{thm:AsympFirstOrder} with a cylindrical container, where the exact solution is well-known; see, {\it e.g.} \cite[pp.764]{Ibrahim:2005aa} or \cite[pp.415]{Reynolds:1966aa}. Consider a solid cylinder with radius $a>0$. Assume that the free surface $\F$ and the flat bottom lies at the plane $\{z=0\}$ and $\{z=-h\}$, respectively, with $h>0$ . Multiplying \eqref{eq:SloshingTension4} by $\omega$ and substituting \eqref{eq:SloshingTension3} yields the simplified system
\begin{subequations}\label{eq:SloshingTensionCylinder}
\begin{alignat}{2}
\Delta\Phi & = 0 && \qquad\textrm{ in }\D, \\
\partial_\n\Phi & = 0 && \qquad\textrm{ on }\B, \\
\Phi_z - \omega^2\Phi & = \frac{1}{\bond}\Big[(\Phi_z)_{xx} + (\Phi_z)_{yy}\Big] && \qquad\textrm{ on }\F.
\end{alignat}
\end{subequations}
Using \emph{separation of variables}, one can compute the explicit solution of \eqref{eq:SloshingTensionCylinder}  in cylindrical coordinates $(r,\theta,z)$ with (upon nondimensionalizing the length with radius $a>0$)
\[ 0\le r\le 1, \ \ 0\le\theta\le 2\pi, \ \ -\frac{h}{a}\le z\le 0. \]

\begin{lemma}
The solution of \eqref{eq:SloshingTensionCylinder} has the form
\begin{subequations}
\begin{align}
\Phi(r,\theta,z) = \sum_{n=0}^\infty\sum_{m=1}^\infty & J_n\left(z_{nm}r\right)\dfrac{\cosh\left(\dfrac{z_{nm}(az+h)}{a}\right)}{\cosh\left(\dfrac{z_{nm}h}{a}\right)}\Big[a_{nm}\cos(n\theta) + b_{nm}\sin(n\theta)\Big] \label{eq:cylinder1}\\
& \omega_{nm}^2= z_{nm}\tanh\left(\dfrac{z_{nm}h}{a}\right)\left[1+\frac{1}{\bond}z_{nm}^2\right], \label{eq:cylinder2}
\end{align}
\end{subequations}
where $J_n(\cdot)$ is the Bessel function of the first kind with order $n$ and $z_{nm}$ is the $m$th root of $J_n'(\cdot)$. In the case where $\bond=\infty$, we recover the eigenvalues $\lambda_{nm}$ for the mixed Steklov-Neumann problem \eqref{eq:Steklov}
\begin{equation}
\lambda_{nm}^2 = z_{nm}\tanh\left(\frac{z_{nm}h}{a}\right). \label{eq:cylinder3}
\end{equation}
For $n=0$, the eigenvalues $\omega_{0m}, \lambda_{0m}$ are simple. 
\end{lemma}

For $n=0$, it is not difficult to verify that the first-order term in the expansion  of \eqref{eq:cylinder2}, 
\[ \omega_{0m}(\varepsilon ) 
= \omega_{0m}(0)  + \varepsilon \omega_{0m}'(\varepsilon )  + o(\varepsilon ) 
= \lambda_{0m} + \varepsilon  \left(\frac{\lambda_{0m}}{2}z_{0m}^2\right) + o(\varepsilon ), 
\]
agrees with  Theorem \ref{thm:AsympFirstOrder}, 
\[ \omega_{0m} =\lambda_{0m} + \varepsilon   \left(\frac{\lambda_{0m}}{2}\frac{\|\nabla_\F\Phi_{0m}\|_{L^2(\F)}^2}{\|\Phi_{0m}\|_{L^2(\F)}^2}\right) + o(\varepsilon ). \]


\section{Discussion} \label{sec:disc}
We have considered the small-amplitude fluid sloshing problem for an incompressible, inviscid, irrotational fluid in a container, including effects due to surface tension on the free surface.  As opposed to the zero surface tension case, where the problem reduces to a partial differential equation for the velocity potential, we obtain a coupled system for the velocity potential and the free surface displacement \eqref{eq:SloshingTension}. In Section \ref{sec:Proof}, we derived a new variational formulation of the coupled problem and establish the existence of solutions using the direct method from the calculus of variations. In the limit of zero surface tension, we recover the variational formulation of the classical mixed Steklov-Neumann eigenvalue problem \eqref{eq:Steklov}, as derived by Troesch, and obtain the first-order perturbation formula for a simple eigenvalue. 

As mentioned in Subsection \ref{sec:ZeroSurfTen}, the location of high spots for the sloshing problem \eqref{eq:Steklov} has been  investigated in two and three dimensions. Some results for specific container geometries   are summarized as follows. 

\begin{enumerate}
\item Consider a trough $W=\D\times (0,l)\subset\R^3$ of length $l>0$, with uniform cross section $\D$.  If the wetted boundary, $\B$, is the graph of a negative $C^2$-function given on $\F$ and $\B$ intersects $\F$ at a  nonzero angle, then the trace $\Phi_1(x,y,0)$ attains  its extrema  only on the boundary of the rectangular free surface of the trough $\partial\F$ \cite{Kulczycki:2011aa}. A similar result for the two-dimensional cross section is given in \cite{Kulczycki:2009aa}

\item Consider a bounded Lipschitz domain $\D$ which is axisymmetric and convex, such that $\D\subset\F\times\{z\in (-\infty,0)\}$. The boundary $\partial\D$ consists of the free surface $\F$ which is a disc of radius $a>0$ and the wetted boundary $\B$. 
If $\Phi_1(x,y,z)$ is odd in the $x$-variable, then the free surface height attains its extrema at $(\pm a,0,0)$ \cite{Kulczycki:2012aa}. 

\item Consider the \emph{ice fishing problem}, where $\D = \R_-^3 = \{(x,y)\in\R^2, z\in(-\infty,0)\}$ with free surface $\F=\{x^2+y^2<b^2, z=0\}$ and wetted boundary $\B=\partial\R_-^3\setminus\bar\F$. It was shown in \cite{Kulczycki:2009aa} that $\Phi_1$ attains its extrema on the interior of $\F$. 
\end{enumerate}

Motivated by the ice fishing problem, in \cite{Kulczycki:2014aa}, axisymmetric, bulbous ($\D\not\subset \F\times\{z\in (-\infty,0)\}$) containers   are studied using finite element methods. It is observed that such domains have fundamental eigenfunctions with high spots which are on the interior of $\F$.  
However, for this container geometry, $\n_{\partial \F} \neq \n_{\B}$ on $\partial \F$ as is assumed in the physical derivation of the contact line boundary condition; see Section \ref{sec:ContactAngle}. 
Because $\bond \to \infty$ is a singular limit, including the physical effects due to surface tension could result in qualitative changes in the 
sloshing modes near $\partial \F$, including the location of high spots.
These questions will be addressed  in forthcoming work using computational methods by investigating eigenfunctions near $\partial \F$  for large but finite $\bond$.

In \cite{Troesch:1965aa}, the variational formulation \eqref{eq:SteklovVF} is used to find the shape of the axisymmetric container with fixed volume that maximizes the fundamental  eigenvalue. In this work, it is assumed that (i) the container is very shallow and (ii) effects due to surface tension are neglected. It would be of interest to extend this work by addressing these two assumptions.

 In \cite{Troesch:1972aa} it is shown that there exist vessel geometries, referred to as  \emph{isochronous containers},  with the remarkable property that the fundamental sloshing  frequency of a fluid is independent of the level to which the container is filled. Such geometries are shown to exist not only for the fundamental mode but for higher  modes as well. In this work, and recent papers which significantly extend this work  \cite{Weidman:2016aa, Weidman:2016bb, Weidman:2016cc}, axisymmetric isochronous containers are found by using the \emph{inverse method of solution}. It would be interesting to include the effect of surface tension in this work.

\subsection*{Acknowledgements} \ We would like to thank James Keener for helpful discussions and comments.


\appendix
\section{Physical Derivation} \label{sec:AppA}
A complete derivation of the nonlinear water wave equations can be found in \cite{Lamb:1932aa, Acheson:1990aa, Ibrahim:2005aa,  Leal:2007aa}. We are concerned with an irrotational flow of an incompressible, inviscid fluid with constant density, occupying a bounded region $\tilde\D\subset\R^3$ in a simply-connected container with a rigid bottom $\B$. Denote by $\tilde \u(\tilde \x,\tilde t)$ the fluid velocity field and $\tilde z=\tilde\eta(\tilde x,\tilde y,\tilde t)$ the displacement of the disturbed fluid free surface from the plane $z=0$.

Irrotationality means $\tilde\u$ has zero curl, which gives existence of a velocity potential $\tilde\phi(\tilde\x,\tilde t)$ such that $\tilde\u = \nabla\tilde\phi=(\tilde\phi_{\tilde x},\tilde\phi_{\tilde y},\tilde\phi_{\tilde z})$. This combined with incompressibility condition shows that $\tilde\phi$ satisfies Laplace's equation
\begin{equation}
\nabla\cdot\tilde\u = \nabla\cdot\nabla\tilde\phi = \Delta\tilde\phi=0\qquad\textrm{ in }\D_T. \label{eq:Derivation1}
\end{equation}
No penetration boundary condition is imposed on the wetted  boundary $\B$, so that
\begin{equation}
\tilde\u\cdot\n_\B = \nabla\tilde\phi\cdot\n_\B = \partial_\n\tilde\phi = 0\qquad\textrm{ on }\B. \label{eq:Derivation2}
\end{equation}
The fluid free surface is an interface between gas and liquid. Such an interface requires two boundary conditions. First, a \emph{kinematic boundary condition} which requires the normal fluid velocity of a fluid particle on the free surface to equal the normal velocity of the free surface itself. This means that fluid particles on the free surface must remain on the free surface. Defining an implicit form $G(\tilde\x,\tilde t) = \tilde z - \tilde\eta(\tilde x,\tilde y,\tilde t)$, it follows that the material derivative of $G$ is 0 at the free surface, {\it{i.e.}}
\begin{alignat}{2}
\quad  \tilde\eta_{\tilde t} + \nabla\tilde\phi\cdot\nabla(\tilde\eta - \tilde z) & = 0 &&  \qquad\textrm{ on }\tilde z=\tilde\eta(\tilde x,\tilde y,\tilde t). \label{eq:Derivation3}
\end{alignat}
Second, a dynamic boundary condition  balances the forces at the free surface. Due to the effects of surface tension, there is a pressure jump across the free surface. Assuming constant surface tension, $T$, the normal stress balance equation has the form
\begin{equation}
p_\textrm{fluid}(\tilde\x,\tilde t) - p_\textrm{atm} = T\nabla\cdot\n_{\F_T}\qquad\textrm{ on }\tilde z=\tilde\eta(\tilde x,\tilde y,\tilde t), \label{eq:Bernoulli1}
\end{equation}
where $\n_{\F_T}$ is the outward unit normal to the free surface $\tilde z=\tilde\eta(\tilde x,\tilde y,\tilde t)$. Equation \eqref{eq:Bernoulli1} can be written in terms of the velocity potential using Bernoulli's principle, which is a reduction of the Navier-Stokes equation for an inviscid fluid.  For unsteady irrotational flow, Bernoulli's principle is
\[ \tilde\u_{\tilde t} = -\nabla\left(\frac{p_\textrm{fluid}}{\rho} + \frac{1}{2}| \tilde \u|^2 + g\tilde z\right). \]
Substituting $\tilde\u = \nabla\tilde\phi$, rearranging, and integrating with respect to time gives
\begin{equation}
\tilde\phi_{\tilde t} + \frac{p_\textrm{fluid}}{\rho} + \frac{1}{2}|\nabla\tilde\phi|^2 + g\tilde z = H(\tilde t), \label{eq:Bernoulli2}
\end{equation}
where $H(\tilde t)$ is an arbitrary function of time only, which we may conveniently choose to be $p_\textrm{atm}/\rho$. The consequence is that upon evaluating \eqref{eq:Bernoulli2} at $\tilde z=\tilde\eta(\tilde x,\tilde y,\tilde t)$ and using \eqref{eq:Bernoulli1} we are left with
\begin{equation}
\tilde\phi_{\tilde t} + \frac{1}{2}|\nabla\tilde\phi|^2 + g\tilde z = -\frac{T}{\rho}\nabla\cdot\n_{\F_T}\qquad\textrm{ on }\tilde z=\tilde\eta(\tilde x,\tilde y,\tilde t). \label{eq:Derivation4}
\end{equation}

Writing the implicit form of the free surface $G(\tilde\x,\tilde t) = \tilde z - \tilde\eta(\tilde x,\tilde y,\tilde t)$ as before, its outward unit normal is given by 
\begin{equation}
\n_{\F_T} = \frac{\nabla G}{|\nabla G|} = \frac{-\tilde\eta_{\tilde x}\X - \tilde\eta_{\tilde y}\y  + \z}{\sqrt{1+\tilde\eta_{\tilde x}^2 + \tilde\eta_{\tilde y}^2}}, \label{eq:Normal1}
\end{equation}
where $\X, \y, \z$ are the unit basis vectors in Cartesian coordinates. Computing $\nabla\cdot\n_{\\F_T}$ gives
\begin{align}\label{eq:Normal2}
\nabla\cdot\n_{\F_T} 
 = \frac{-(\tilde\eta_{\tilde x\tilde x} + \tilde\eta_{\tilde y\tilde y}) - (\tilde\eta_{\tilde x\tilde x}\tilde\eta_{\tilde y}^2 + \tilde\eta_{\tilde y\tilde y}\tilde\eta_{\tilde x}^2) + 2\tilde\eta_{\tilde x}\tilde\eta_{\tilde y}\tilde\eta_{\tilde x\tilde y}}{(1+\tilde\eta_{\tilde x}^2 + \tilde\eta_{\tilde y}^2)^{3/2}}. 
\end{align}
The contact-line boundary condition, 
\begin{equation}
0 = \n_\B\cdot\n_{\F_T} \qquad\textrm{ on }\partial\F_T, \label{eq:Derivation5}
\end{equation}
is derived in Section \ref{sec:ContactAngle}. 
We then nondimensionalize the system of PDEs \eqref{eq:Derivation1}, \eqref{eq:Derivation2}, \eqref{eq:Derivation3}, \eqref{eq:Derivation4}, \eqref{eq:Derivation5} with dimensionless variables in \eqref{eq:NonDim}, which results in the \emph{nonlinear sloshing problem with surface tension} \eqref{eq:Slosh1}. 

We consider an equilibrium solution $(\phi_0,\eta_0)=(c,0)$ of \eqref{eq:Slosh1}, where $c$ is any constant scalar function (which gives zero velocity field). Assuming the free surface displacement $\eta$ is a small perturbation of $\{z=0\}$,  we look for solutions of the form 
\begin{equation*}
\phi(x,y,z,t) = c + \varepsilon\hat\phi(x,y,z,t) 
\quad \textrm{and} \quad 
\eta(x,y,t) = \varepsilon\hat\eta(x,y,t), 
\end{equation*}
where $\varepsilon>0$ is some small parameter and collect  $\O(\varepsilon)$ terms. 

Next, we Taylor expand $\hat\phi$ and its derivatives around $z=0$. This transforms the boundary conditions,  \eqref{eq:Slosh1c} and \eqref{eq:Slosh1d}, from $\F_T$ to $\F$.  Consequently, the time-dependent  linearized problem for \eqref{eq:Slosh1} has the form
\begin{subequations}\label{eq:Slosh2}
\begin{alignat}{2}
\Delta\phi & = 0 && \qquad\textrm{ in }\D, \\
\partial_\n\phi & = 0 && \qquad\textrm{ on }\B, \\
\eta_t & = \phi_z && \qquad\textrm{ on }\F, \\
\phi_t + \eta & = \frac{1}{\bond}(\eta_{xx} + \eta_{yy}) && \qquad\textrm{ on }\F, \\
\partial_\n\eta & = 0 && \qquad\textrm{ on }\partial\F, 
\end{alignat}
\end{subequations}
where $\phi_t,\eta_t$ denotes the partial derivative of $\phi,\eta$ with respect to time $t$. 

Finally, \eqref{eq:SloshingTension} is obtained by  seeking time harmonic solutions (with angular frequency $\omega$ and phase shift $\delta$) via the ansatz $ \hat\phi(x,y,z,t) = \Phi(x,y,z)\cos(\omega t + \delta)$ and $ \hat\eta(x,y,t) =\xi(x,y)\sin(\omega t + \delta),$ where $\Phi(x,y,z)$ and $\xi(x,y)$ are the sloshing velocity potential and height, respectively. 

\clearpage
{
\bibliographystyle{siamplain}
\bibliography{reference.bib} }

\end{document}